\documentclass{article}
\usepackage[utf8]{inputenc}
\usepackage[english]{babel}
\usepackage{polski}
\usepackage[colorlinks=true, citecolor=blue, linkcolor=blue, linktocpage=true]{hyperref}

\newcommand{\ing}{P}

\newcommand{\ATC}{\mathsf{ATC}}
\newcommand{\at}{\mathfrak{at}}
\newcommand{\pl}{\mathsf{pl}}

\newcommand{\are}{\preccurlyeq}

\usepackage{amsmath}
\usepackage{amsthm}
\usepackage{amsfonts}
\usepackage{amssymb}
\usepackage{amsbsy}
\usepackage{latexsym}
\usepackage{hyperref}

\usepackage{amsmath}

\usepackage{breqn}

\newtheorem{theorem}{Theorem}
\newtheorem{lemma}{Lemma}
\title{Atomism Axiomatised Using 
Mereological Composition as a Primitive Notion\footnote{I would like to express my sincere gratitude to Achille Varzi for his time and effort in providing his expertise and guidance, which have been instrumental in shaping the final outcome of this work.
}}
\author{Marcin Łyczak}
\date{The work was written during my scholarship at Columbia University under the supervision of Achille Varzi (16.11.2022 - 8.12.2022) and was submitted to Synthese on 16.03.2023.}
\usepackage{hyperref}
\begin{document}
\maketitle
\begin{abstract}
Atomism is the view that \textit{everything is composed of atoms}. The
view within the framework of the contemporary formal approach is expressed on
the ground of mereology with the use of the primitive notion of
\textit{being a part} as \textit{every object has at least one atomic part} \cite[48]{CaVa99}, \cite[145]{CoVa21}, \cite[42]{Si87}, or using \textit{mereological fusion}
\cite{Sh14, Va17} which is defined by being a part. We will briefly present a
discussion between A. Varzi and A. Shiver concerning the two approaches and
propose a new intuitive axiomatic characterization of atomism. We build a system with a primitive notion of composition that holds between individuals and
pluralities. We assume only two specific axioms: \textit{each object is a unique composition
of unique atoms}, and \textit{being a composition of some objects is
equivalent to being the composition of all atoms of these objects}.
 In our approach, notions of \textit{part}, and \textit{atom}, are secondary to composition: atom is defined as an object that cannot be a composition of
two or more objects, and part is defined as inclusion between atoms.
We will show that the theory, with only these two
specific axioms, is sufficient to adequately express atomism as we prove that the theory is definitionally equivalent to \textit{atomistic
extensional mereology} with plural quantification allowed and
mereological fusion defined. Our theory requires neither full
\textit{comprehension schema} nor the existence of any specific compositions (there are models with only atoms). Therefore, it may constitute the basis on which atomistic concepts of reality are further strengthened, specifically those in
which the notion of being a part is not used. To the best of our knowledge, our proposal is the only formal theory of atomism using the primitive notion of composition.
\end{abstract}
\section*{Introduction}
In philosophy, there are many different kinds of atomism. Here, we understand atomism in a very general sense: as a view that claims that everything that \textit{exists} is either a simple object or is composed of simple objects, with the assumption that something \textit{exists}. Atomism is justified by different theories, where we find various interpretations of the statement
\begin{equation}
\tag{$\mathsf{AT1}$}\label{TA1}
\mbox{\textit{Everything is composed of atoms}}.
\end{equation}
Atomism within the framework of the contemporary formal approach is usually
expressed on the ground of mereology, which is a formal theory of parts
and wholes, originally formulated by S. Leśniewski. The central notion of mereology is mereological
\textit{composition} which is usually defined with the use of the primitive notion of
\textit{being a part}/\textit{proper part}. These notions are appropriately specified in different versions of mereology. Loosely speaking, the mereological composition of given
objects is a concrete whole composed of all of them, and only them. In contemporary literature \cite[48]{CaVa99}, \cite[145]{CoVa21}, \cite[42]{Si87}, however, 
atomism is formulated without using the notion of composition as
\begin{equation}
\tag{$\mathsf{AT2}$} \label{TA2} \mbox{\textit{Every object has at least
one atomic part}}.
\end{equation}
A. Shiver criticizes articulating atomism as \eqref{TA2} and argues
that \eqref{TA1} should be expressed using mereological fusion
\cite{Sh14}. However, as A. Varzi shows, 
having only assumptions of reflexivity and the transitivity of being
a part and the definition of mereological composition: \textit{sum} or \textit{fusion}, formal counterparts of \eqref{TA1} and \eqref{TA2} are equivalent already in classical logic 
\cite{Va17}. As Varzi pointed out, the main problem of the models
considered by Shiver against \eqref{TA2} does not lie in the wrong
formulation of atomism by \eqref{TA2}, but that being a part is not well
founded.\footnote{Currently, more powerful variants of atomism called `superatomisms' have been proposed by Cotnoir  \cite{Co13}. In superatomism, the proper part relation is
well-founded i.e.
no descending chain of a proper parthood relation can be infinite.
This is being analyzed and developed \cite{Di20, Uz17}. On the other
hand, \textit{gunky} models of atomism are also being considered \cite{Ts15}. In models of this kind every object is composed of atoms but some objects
may be divided infinitely into smaller and smaller parts, and all of
these parts have an uncountable number of their own parts. We follow Varzi's view 
that \textit{atomism is a thesis about
composition, not decomposition} \cite{Va17} and direct our attention here to atomism in
a broad sense as just \eqref{TA1}.} We believe that Varzi is formally right, at the same time we agree with Shiver in that \eqref{TA2} does not intuitively express the atomistic position. 
In defend of \eqref{TA2} at least two issues can be indicated: \eqref{TA2} is expressible in first-order logic, and it is not necessary to decide how mereological composition is defined and what its properties are. However, that from an intuitive point of view, mereological composition very clearly express atomism. 
Shiver expresses atomism using defined notion of mereological fusion that is secondary to being a part. We propose reformulating Shiver's approach and take mereological composition as the only primitive notion.\footnote{Let us note that there are philosophical arguments for treating mereological composition as primitive \cite{Fi10, Kl17}, but they need not relate to \eqref{TA1} and so we will not discuss them here.} We aim to:
\begin{itemize}\setlength\itemsep{0em}
\item find a formal system with a primitive notion $F$ of composition that holds between individuals, and pluralities;
\item show that with this notion the standard mereological notions of part and atom are definable, and primitive operator $F$ can be understood as a standard composition;
\item show that the system is definitionally equivalent to standard system of extensional atomistic mereology, in which plural quantification is allowed.
\end{itemize} 
We begin with some remarks on mereological composition, atomistic mereology, and the discussion between Shiver
and Varzi (\ref{1}).
Next, we provide new axiomatics for atomistic extensional mereology, where a formal counterpart of (\ref{TA1}) is an axiom, and the only primitive notion is composition (\ref{2}). Finally, we prove definitional equivalence of the presented system with the standard axiomatization of  \textit{atomistic extensional mereology} in which plural quantification is allowed  (\ref{3}).
\section{Mereological composition, extensional mereology and atomicity}\label{1}
The first version of mereology was given by Leśniewski
in a natural language \cite{Le16} (trans. in
\cite{Le92}). He later presented a mature version of mereology
on the ground of ontology, which is a theory of objects, expressed in a
first-order identity-free language. The only primitive notion of his
ontology was a two-place predicate $\varepsilon$ read `is', applied to two names of the same category. 
Leśniewski added to his system a primitive name-forming operator $pt$ applied to names, read
`part of'. The atomic expression $x\varepsilon pt(y)$ is read `$x$ is a part of $y$'. The original Leśniewski's ontological approach to mereology is described by R. Urbaniak in \cite{Ur14}. A. Tarski extracted mereology from its ontological context and treated
mereology as a theory of mathematical relational structures with the primitive relation of being a part, and used this as a foundation for point-free geometric
considerations \cite{Ta56}. 
Mereology today is very willingly treated as just first-order theory with identity, with the
only primitive predicate for `is a part/proper part'. This approach was described by P. Simons in \cite{Si87}, and more recently by Varzi and Cotnoir \cite{CoVa21}. While mentioned approaches differ, they share a common core \cite{VaStanford}. We focus on mereology which allows plural quantification described in \cite[238-245]{CoVa21}.
Mereology with plural quantification has greater expressive power than first-order mereology, what is used to formal analyzes of atomism \cite{Di20, Gi23}. Quantification over pluralities allows us to say that for any object, there exist atoms from which the object is composed, and this is crucial for our characterization of atomism. As we mentioned, the mereological composition of given objects is a concrete of all of them, and only them, put together. Mereological composition is a subject of actual philosophical studies concerning, 
inter alia, its extensionality, the relationship between wholes and their parts, the existence of arbitrary compositions, and the nature of how objects are put together \cite{CaLa21}. 
In the original mereology, to
say that something is a mereological composition of some objects Leśniewski used  $x\varepsilon Kl(y)$. What is important is that in the expression
$x\varepsilon Kl(y)$ both $x$ and $y$ are variables of the same category. In
modern first-order approaches, counterparts of non-empty names used by
Leśniewski are constructed with the use of formulas with at least one free variable. To
say that $x$ is a mereological composition of objects that are $\varphi$,
$F_{\varphi_{y}}x$ is used, where $y$ is a free variable in
the formula $\varphi$. In the case of mereology with plural
quantification, there is no need to use formulas with free variables, because plural variables play their role. The mereological
composition is expressed then as $F_{zz}x$, where $zz$ is a plural variable and $x$ is individual variable. In the
case of mereological relational structures, the notation $x F X$ is used, where $F$ is a relation between an element $x$ belonging to the domain of a given structure  and $X$ is a distributive subset of the domain.

Shiver formulates atomism using plural quantification and defined mereological fusion. 
We follow the nomenclature from \cite[238-245]{CoVa21} and use individual variables: $x, y, z$, and plural ones: $xx, yy, zz,...$ . The logical symbols are classical connectives, quantifiers, identity $=$ applied to
individual variables, and $\prec$ applied to individual and plural variables on the left and right side, respectively. Expression $x\prec yy$ is read `$x$ is one of $yy$'s'. The only non-logical symbol is $P$, for `is a part of' applied to individual variables. We assume as a base only axioms of two sorted logic and axioms for first-order identity. 
Following Varzi \cite{VaStanford}, we call
\textit{extensional mereology} the theory which has the following theses:
\begin{gather}
     \tag{$\mathtt{ref}$} \label{ring} P xx,\\
\tag{$\mathtt{trans}$} \label{transing} Pxy \land P yz \to Px z,\\
\tag{$\mathtt{ants}$} \label{antising} Pxy \land Pyx \to x=y,\\
\tag{$\mathtt{ssp}$}\label{ssp} \neg Px y \to \exists z (P z x \land
\neg O z y ),
\end{gather}
where $O$ is a predicate for `overlapping' defined as 
\begin{equation}
\tag{$O$} \label{defO} O x y := \exists z (P z x\land P
zy).
\end{equation}
We denote the set of theses of extensional mereology with plural quantification as $\mathsf{EM}_{\pl}$ and we use the following notation
$$\mathsf{EM}_{\pl}=\eqref{ring}+\eqref{antising}+\eqref{transing}+\eqref{ssp}.$$
Let us briefly comment extensional mereology. Formulas \eqref{ring},  \eqref{transing}, and \eqref{antising} taken as axioms state that the semantic correlate of predicate $P$ is a \textit{partial order relation}. This is the core of
mereological theories which describe being a part in an inclusive sense. Let us focus for a moment on \eqref{ssp} and its properties. Formula \eqref{ssp} is called the \textit{strong supplementation principle} and was considered in the context of mereology by Simons \cite[28-29]{Si87}, however its even stronger version was used as a mereological axiom already by Leśniewski \cite{SwLy20}. \eqref{ssp} is a reinforcement of widely discussed in philosophical literature \cite{Co21} the \textit{weak supplementation principle}
\begin{equation}
\tag{$\mathtt{wsp}$}\label{wsp} PPx y \to \exists z (P z y \land
\neg O z x),
\end{equation}
where $PP$ is a predicate for `proper part' defined as 
$$PPxy:= Pxy\land \neg x=y.$$
The axiom \eqref{ssp} is stronger than \eqref{wsp} because assuming partial order axioms for $P$ \eqref{ssp} implies \eqref{wsp}; the converse does not hold. Moreover, what is more important for us, \eqref{ssp} added to partial order axioms for $P$ allows one to prove (a): the equivalence of two formulations of mereological composition: sum and fusion, which we will introduce soon, as well as ensuring that (b): if there is a given sum, then it is unique, and (c): if there is a given fusion, then it is unique. Of course, (a) implies (b)$\leftrightarrow$(c). However, weak supplementation principle \eqref{wsp} added to partial order axioms for $P$ allows to prove only (b) (see e.g. \cite{Lo21}). Thus, \eqref{ssp} is convenient for building weak mereological theories that do not determine which sums/fusions exist, while still guaranteeing their generally wanted properties. 

We introduce a definition of Leśniewski's mereological sum $F_{zz}x$, read `$x$ is a mereological sum of $zz$'s'; and a definition of mereological fusion $F^{\star}_{zz}x$, read `$x$ is a mereological fusion of $zz$'s' in the following way:
\begin{gather}
\tag{$\mathtt{Df}.F$} \label{defF} F_{zz} x \leftrightarrow \forall_{z\prec zz}Pz x \land \forall y(P y x \to \exists_{z\prec zz} O z y),\\
         \tag{$\mathtt{Df}.F^{\star}$} \label{deffu} F^{\star}_{zz}x \leftrightarrow \forall y (Oyx \leftrightarrow \exists_{z\prec zz } O z y).
     \end{gather}
In \textit{general extensional mereology}, after Leśniewski, it is assumed that for any objects
represented by any non-empty general name there exists (unique) mereological composition of
them \cite{Ho08}.
Regardless of the definition of composition and its existential assumptions, an \textit{atom} is any object that has no proper parts, which we define as
\begin{gather}
\tag{$\mathtt{Df}. A_{P}$}\label{dfAP} Ax \leftrightarrow \forall y (Pyx \to x=y).
\end{gather}
Using the above we can express \eqref{TA2} as
\begin{equation}
\tag{$\mathtt{AT2}_{\ing}$}\label{TA2P} \forall x \exists y
(P y x \land  Ay).
\end{equation}
To express atomism, which we want to discuss here, Shiver uses the notion of fusion.
Using plural quantification, he define plural constant $aa$ for `atoms', which we introduce as an instantiation of comprehension schema
\begin{equation}
\tag{$\mathtt{Df}.aa_{P}$} \label{aaP}
x\prec aa\leftrightarrow Ax.
\end{equation}
Moreover, Shiver uses the predicate $S$ for `being a fusion of atoms'
\begin{equation}
\tag{$S$}\label{S} Sx := \exists_{yy\are aa} F^{\star}_{yy}x,
\end{equation}
where $\are$ is many-many predicate read as `$xx$'s are among $yy$'s' that we define as
\begin{equation}
\notag
xx\are yy := \forall z(z\prec xx \to z\prec yy) \land \exists x (x\prec xx).
\end{equation}
Atomistic standpoint \eqref{TA1} my be expressed as 
$$\forall x (Sx).$$
Shiver expresses the atomistic standpoint in slighty different way, which he calls \textit{general atomicity}, by the formula
$$\forall x \exists y (Pxy \land Sy).$$
This approach is analyzed by Varzi in first-order mereology in
\cite{Va17} (and later also in \cite[145-147]{CoVa21} and \cite{Gi23}). We express Varzi analysis in mereology with plural quantification, because  we want to be consistent with earlier and later consideration. We take the notion of atoms 
as Shiver does, and we introduce a notion of
\textit{atomic parts of}
\begin{equation} \tag{$\mathtt{Df}.\at^{\mathsf{ind}}_{P}$} \label{atP}y\prec
\at_{x}\leftrightarrow Py x \land y\prec aa.
\end{equation}
Following Varzi's considerations, having only \eqref{ring} and \eqref{transing}, axiom \eqref{TA2P} implies
$$\forall x(\forall y (Oyx\leftrightarrow \exists_{z\prec
\at_{x}}O yz)),$$ 
and
$$\forall x(\forall_{y\prec \at_{x}} Py x\land \forall y(P y x\to
\exists_{z\prec \at_{x}} O yz)).$$
These two formulas, applying the definitions of mereological fusion \eqref{deffu} and sum \eqref{defF}, yields
\begin{equation}
\tag{$\mathtt{AT1}_{\mathcal{F}}$}\label{TA1F}
\forall x (\mathcal{F}_{\at_{x}}x), \mbox{ for } \mathcal{F}\in \{F^{\star}, F\}.
\end{equation}
The above schema expresses atomism \eqref{TA1} in both senses of mereological composition: fusion and sum. \smallskip\\
From \eqref{TA1F} and the reflexivity of being a part \eqref{ring} we infer
$$\forall x \exists y \exists_{zz\are aa} (Pxy \land \mathcal{F}_{zz}y), \mbox{ for } \mathcal{F}\in \{F, F^{\star}\}.$$
When $\mathcal{F}$ is $F^{\star}$, we use definition \eqref{S} and obtain Shiver's formula for general atomicity.

Following Varzi \cite{VaStanford} we call \textit{atomistic extensional mereology} a theory which is extensional mereology and has as a thesis \eqref{TA2P}. We denote atomistic extensional mereology with the definition of mereological sum \eqref{defF} allowing plural quantification by $\mathsf{AEM}_{\pl}$
$$\mathsf{AEM}_{\pl}=\mathsf{EM}_{\pl}+\eqref{TA2P}+\eqref{dfAP}+\eqref{defF}.
$$
As we mentioned, in $\mathsf{EM}_{\pl}$ formulas characterizing mereological fusion \eqref{deffu} and mereological sum \eqref{defF} are equivalent. So, in $\mathsf{AEM}_{\pl}+\eqref{aaP}+\eqref{atP}$ it does not really matter which
definition of composition we take: \eqref{TA2P} implies \eqref{TA1F},
and the converse implication follows just from the definition of mereological sum.

Before we move on to the main considerations, we would like to discuss
the issue of the primitive notions that are used to axiomatize
mereology.
As we have said, popular approaches to mereology use the primitive notion of
being a part/proper part. Possible axiomatizations with this notions is still being
discussed \cite{CoVa18, Ho08, Va19}. The rich catalog of axiomatizations with the primitive notions of \textit{being external},
\textit{overlapping}, as well as attempts to base mereology on other primitive concepts is given in \cite{So84}.
In the latter mereology is expressed on the basis of Leśniewski's ontology.
The only known axiomatics with the primitive notion of mereological composition is Lejewski's single axiom for general extensional mereology \cite[222]{So84}. In this case, the quantification over function symbols is used, which is allowed in the full original version of Leśniewski's  ontology. Thus, this axiom is not translatable into the language of modern mereologies, in particular also into the language that we use in our approach. We also believe that our axiomatics has the advantage of being more intuitive than Lejewski's equivalence axiom, which in the Polish notation has seventy-seven symbols.
\section{Theory \texorpdfstring{$\ATC$}{LG}. Axiomatization} \label{2}
Now we formulate an axiomatic theory of atomism in which the primitive notion is composition. We do not assume the existence of any specific
compositions: our theory is open to axiomatic strengthening as needed.
We express it on the basis of a very small fragment of the logic of plurals $\mathsf{PFO}$ \cite[15-19]{FlLi21} without
comprehension schema, and without assumption of non-emptiness of plurals. In other words, in our proposal, one can accept that axioms or reject them.\smallskip\\
The only non-logical symbol of our theory is $F$ for composition, used in the context with individual variables, and
plural terms. Plural terms are: plural variables, $aa$ for `atoms', and
$\at_{\pmb{xx}}$ for `atoms of which
$\pmb{xx}$'s are composed of'. 
\smallskip\\
The following two are specific axioms of our theory:
\begin{gather}
\tag{$\mathtt{ATC1}$}\label{AC1} \forall x \exists_{zz\are aa}
(F_{zz}x \land \forall_{yy\are aa}(F_{yy}x\leftrightarrow zz\approx yy)\land \forall
y (F_{zz}y\to x=y)),\\
\tag{$\mathtt{ATC2}$}\label{AC2} F_{zz}x\leftrightarrow F_{\at_{zz}}x,
\end{gather}
where :
\begin{gather}\tag{$\mathtt{Df}.aa_{F}$} \label{dfaa}
x\prec aa\leftrightarrow \forall yy\forall_{z\prec yy} (F_{yy}x\to
z=x),\\
\tag{$\mathtt{Df}.\at^{\pl}_{F}$} \label{dfpl} x\prec
\at_{zz}\leftrightarrow \exists_{y\prec zz} \exists_{yy\are
aa}(F_{yy}y \land x\prec yy ),\\
\tag{$\are$}\label{are}
xx\are yy := \forall z(z\prec xx \to z\prec yy)\land \exists x x \prec xx\\
\tag{$\approx$}\label{dfapprox} zz\approx yy := \forall x (x\prec zz \leftrightarrow x \prec yy).
\end{gather}
Let us briefly comment on our axiomatization. \eqref{AC1} states that
everything is a unique composition of unique atoms.
\textit{Atom} is defined as an object that cannot be a composition of
two or more objects, and \textit{atoms of} $\pmb{zz}$'s
are defined as the sum of all atoms of all individuals that are $\pmb{zz}$.
The axiom \eqref{AC1} is intended to capture \eqref{TA1}, but
\eqref{AC1} itself does not guarantee that $F$ is the mereological sum
in Leśniewski's sense. This is why we need one more axiom. \eqref{AC2}
states that the composition of some objects is equivalent to the composition of the atoms of these objects. Formulas \eqref{dfaa} and \eqref{dfpl} are instantiations of comprehension schema of axioms.\footnote{In $\mathsf{PFO}$, the comprehension schema is the following: $\exists x(A(x)) \to \exists zz \forall x(x \prec zz \leftrightarrow A(x))$, because in $\mathsf{PFO}$, all plurals must be non-empty. We do not assume either the full comprehension schema or that every plural must be non-empty. Our instantiations of comprehension schema \eqref{dfaa} and \eqref{dfpl} do not have in the predecessor non-emptiness assumption. However, our axioms guarantee that $aa$ is non-empty, and if $F_{zz}x$, then also $zz$ and $\at_{zz}$ are non-empty.} Pluralities $aa$ and $\at_{\pmb{xx}}$ are not primitive, as they are defined using the primitive mereological composition $F$.
\smallskip\\
We denote our axiomatic theory of atomistic compositions as $\ATC$, so
$$\ATC=\eqref{AC1}+\eqref{AC2}+\eqref{dfaa}+\eqref{dfpl}$$
As we show, this axiomatic theory of atomism is definitionally equivalent to atomistic extensional mereology with plural quantification.
\section{Definitional equivalence of \texorpdfstring{$\ATC$}{LG} and
\texorpdfstring{$\mathsf{AEM}_{\pl}$}{LG}}\label{3}
From axiom \eqref{AC1} we know that each individual there is exactly one plurality that is a unique composition of unique atoms
$$ \forall x \exists^{1}zz (zz\are aa\land 
(F_{zz}x \land \forall_{yy\are aa}(F_{yy}x\leftrightarrow zz\approx yy)\land \forall
y (F_{zz}y\to x=y)),$$
where $\exists^{1}zz A(zz) \leftrightarrow \exists zz A(zz) \land \forall xx \forall yy (A(xx)\land A (yy)\to xx\approx yy)$. \smallskip\\
We can, thus, introduce one place operation that to every $x$ attributes the plurality of all atoms of $x$. For convenience, we use in $\ATC$ the same constant symbol $\at$ as it is used for the atoms of pluralities. We have 
\begin{equation}
\tag{$\mathtt{Df}.\at_{F}^{\mathsf{ind}}$}\label{defatx}
\at_{x}\are aa \land F_{\at_{x}}x \land \forall_{yy\are aa}(F_{yy}x\leftrightarrow yy\approx\at_{x})\land \forall y (F_{\at_{x}}y\to x=y).\end{equation}
Let us note, that theorem $\forall x (\at_{x}\are aa)$ guarantee that pluralities $aa$, and $\at_{x}$ are always non-empty.\smallskip\\
To prove definitional equivalence, we conservatively extend $\ATC$ by
a definition of being a part. We do this using many-many inclusion among
atoms of two given objects as
\begin{equation}
\tag{$\mathtt{Df}.P_{F}$}\label{dfP} Pxy \leftrightarrow  \at_{x}\are \at_{y},
\end{equation}
which may be read as: $x$ is a part of $y$ iff all atoms of $x$ are atoms of $y$.\smallskip\\
Now we turn to the main goal. We formulate proofs using
natural deduction. First, we note that $\ATC$ characterizes being a part as a partial order:
\begin{lemma}\label{lemmaPO}
\eqref{ring}, \eqref{transing}, and \eqref{antising} are provable in
$\ATC+\eqref{dfP}$.
\end{lemma}
 \begin{proof}
To obtain \eqref{ring} we note that from \eqref{defatx} we have $\exists x (x\prec \at_{x} )$, thus $\at_{x}\are \at_{x}$, from reflexivity of $\are$ i.e. $Pxx$. \eqref{transing} we obtain directly from transitivity of $\are$. To prove
\eqref{antising} we assume $\ing x y \land \ing y x$ and by \eqref{dfP} we
obtain $\at_{x}\are \at_{y}\land \at_{y}\are \at_{x}$ i.e. $\at_{x}
\approx \at_{y}$, by \eqref{dfapprox}. We have
$F_{\at_{y}}y$ from \eqref{defatx}, so using it with $\at_{x} \approx \at_{y}$ we obtain and  $F_{\at_{x}}y$, by \eqref{defatx}, and next again using \eqref{defatx} we get $x=y$.
\end{proof}
\noindent In the proofs of \eqref{ssp} and \eqref{defF} we use the following $\ATC+\eqref{dfP}$ theses:
\begin{gather}
\tag{$\mathtt{T1}$} \label{T1} z\prec aa \land  z\prec \at_{x}\to
\at_{z}\are \at_{x},\\
\tag{$\mathtt{T2}$} \label{T2} z
\prec aa \land Ozy \to z \prec \at_{y}.
\end{gather}
\begin{proof}
For \eqref{T1} we assume $z\prec aa \land z\prec \at_{x}$, and we show $\at_{z}\prec \at_{x}$. $\at_{z}$ is non-empty, so fix any $y\prec \at_{z}$. Then from $y\prec \at_{z}\land
z\prec aa\land F_{\at_{z}}z$ using
the definition of atoms \eqref{dfaa} we have $y=z$. The latter with $z\prec \at_{x}$ yields
$y\prec \at_{x}$, so $\at_{z}\are \at_{x}$.\\
For \eqref{T2} we assume $z \prec aa$ and 
$Pcz\land Pcy$ for some $c$, and we prove $z\prec \at_{y}$. From \eqref{defatx} we
have $F_{\at_{z}}z$ and so using $z\prec aa$ and identity we have
$z\prec\at_{z}\land \forall y(y\prec \at_{z}\to y=z)$, so the only atom
of $\at_{z}$ is $z$. Using this and $\at_{c}\are \at_{z}$ taken from
$Pcz$ and \eqref{dfP} we obtain $z\prec \at_{c}$. The latter with
$\at_{c}\are \at_{y}$ taken from $Pcy$ and \eqref{dfP} yields $z\prec
\at_{y}$.
\end{proof}
We take the definition \eqref{defO} of overlapping as in
$\mathsf{AEM}_{\pl}$ and we show that the strong supplementation
principle is a thesis of $\ATC+\eqref{dfP}$.
\begin{lemma}
\eqref{ssp} is provable in $\ATC+\eqref{dfP}$.
\end{lemma}
\begin{proof}
Assume that $\forall z (Pzx \to Ozy)$. We aim to show that $Pxy$, that is, \linebreak[4]
$\forall _{u} (u\prec \at_{x} \to u \prec \at_{y}$). Thus, fix an arbitrary $u \prec \at_{x}$. Thus, since $u\prec aa$  $u \prec \at_{x}$ gives $\at_{u} \are \at_{x}$ by \eqref{T1} i.e., $Pux$, by \eqref{dfP}. So $Ouy$ by the assumption. We have $u\prec aa$ and $Ouy$ thus $u\prec \at_{y}$ by \eqref{T2}. In consequence we have $\at_{x}\are \at_{y}$, i.e. $Pxy$. 
\end{proof}
Now we have to prove that notion of composition axiomatized in
$\ATC$ is the mereological sum in the Leśniewski sense, i.e. we have
to show that \eqref{defF} is a thesis of $\ATC$.
\begin{lemma}
$\forall_{z\prec zz} Pz x \land \forall y(P y x
\to \exists_{z\prec zz}  O z y )\to F_{zz}x$ is a thesis of
$\ATC+\eqref{dfP}$. \end{lemma}
\begin{proof}
We assume $\forall z (z\!\prec\! zz \to Pz x )\!\land\! \forall y(P y x
\to \exists z (z\!\prec\! zz \land  O z y ))$. From $\forall z
(z\!\prec\! zz \to Pz x )$ by \eqref{dfP} we obtain $\forall z
(z\!\prec\! zz \to \at_{z}\are \at_{x})$ thus by \eqref{dfpl} and
\eqref{defatx} we have $\at_{zz}\are \at_{x}$. We assume additionally
that $\at_{zz}\not =\at_{x}$, so $c\prec \at_{x}\land \neg c \prec
\at_{zz}$. From $c\prec \at_{x}$, \eqref{dfaa}, and \eqref{defatx} we
obtain $\at_{c}\are \at_{x}$, so $Pcx$ by \eqref{dfP}. Thus, from
assumption, we have $\exists z (z\!\prec\! zz \land  O z c)$. This with
$c\prec aa$, symmetry of $O$ and (t2) from earlier lemma yields $\exists
z (z \prec zz \land c\prec \at_{z})$, $\forall z (F_{\at_{z}}z)$ we have
from \eqref{defatx} thus $c\prec \at_{zz}$ by \eqref{dfpl} which is
false by assumption so $\at_{zz} =\at_{x}$. As always, we have
$F_{\at_{x}}x$ from \eqref{defatx}, so using $\at_{zz}
=\at_{x}$ and \eqref{defatx} we get $F_{\at_{zz}}x$ and so using `$\leftarrow$' of
\eqref{AC2} we finally obtain $F_{zz}x$ which we wanted to prove.
\end{proof}
\begin{lemma} \label{lemmaFUSU}
$ F_{zz}x \to \forall_{z\prec zz} Pz x \land
\forall y(P y x \to \exists_{z\prec zz}  O z y)$ is a
thesis of $\ATC+\eqref{dfP}$. \end{lemma}
\begin{proof}
We assume $F_{zz}x$ and proceed indirectly. From `$\to$' of \eqref{AC2}
we obtain  $F_{\at_{zz}}x$. Using \eqref{defatx} we get
$\forall_{yy}(F_{yy}x\to yy\approx \at_{x})$, so we take $\at_{zz}/yy$
and we obtain $\at_{x}\approx\at_{zz}$. If $\neg\forall z (z\prec zz \to
Pz x)$, then we take $c/z\colon c\prec zz \land \neg P cx$. From $c\prec
zz$ we obtain that $\at_{c}\are \at_{zz}$ by \eqref{dfpl} and
\eqref{defatx}. The latter with $\at_{x}\approx\at_{zz}$ yields
$\at_{c}\are \at_{x}$, i.e. $Pcx$, which is false by assumption. Now, if
$\neg \forall y(P y x \to \exists z (z\!\prec\! zz \land  O z y ))$,
then we take $c/y$: $P c x \land  \forall z (z\!\prec\! zz \to  \neg O z
c )$. From $P c x$ we have $\at_{c}\are \at_{x}$ and from \eqref{AC1} we have
that some object is in atoms of both $c$ and $x$: $d\prec\at_{c}\land
d\prec\at_{x}$. From $d\prec \at_{x}$ and $\at_{x}=\at_{zz}$ we
have $d\prec \at_{zz}$. From the latter and \eqref{dfpl} there is
$e\prec zz$ such that $d\prec \at_{e}$. From $\forall z (z\!\prec\! zz
\to  \neg O z c )$ with $z/e$ we get $e\prec zz \to  \neg O e c$, so
$\neg O e c$, i.e. $\forall y(Pye \to \neg Pyc)$ so using \eqref{dfP}
and $d/y$ we have $\at_{d}\are \at_{e}\to  \neg \at_{d}\are \at_{c}$. We
have $d\prec \at_{e}$ and $d\prec aa$, so by \eqref{defatx} we have
$\at_{d}\are \at_{e}$ and thus $\neg \at_{d}\are \at_{c}$ but we have
$d\prec \at_{c}$ so with $d\prec aa$ and \eqref{defatx}  yields
$\at_{d}\are \at_{c}$.
\end{proof}
\noindent Now we conservatively extend $\ATC$ by adding the predicate
for being an atom
\begin{equation}
\tag{$\mathtt{Df}.A_{F}$}\label{AF}Ax \leftrightarrow x\prec aa.
\end{equation}
It is clear that in $\ATC$, formula \eqref{TA2P} is a thesis
because of \eqref{AC1} and \eqref{dfaa}.
Moreover, having \eqref{defF} in $\ATC$ we can easily prove that
equivalence \eqref{aaP} is a thesis, and \eqref{atP} may be proved with
the use of \eqref{defatx}. Thus, using lemmas
\ref{lemmaPO}-\ref{lemmaFUSU} we obtain that $\mathsf{AEM}_{\pl}$ is a
subtheory of $\ATC$ conservatively extended by definitions:
\eqref{dfP} of being a part and \eqref{AF} of predicate $A$:
\begin{theorem}\label{AEMAC}
$\mathsf{AEM}_{\pl}+\eqref{aaP}+\eqref{atP}\subseteq\ATC+\eqref{dfP}+\eqref{AF}$.
\end{theorem}
Now we are going to prove converse dependency. First, we show that the
mereological sum in $\mathsf{AEM}_{\pl}$ is extensional
\begin{lemma}\label{lemextF}
$\forall xx \forall yy\forall z(xx\approx yy\land F_{xx}z\to F_{yy}z)$ is provable in $\mathsf{AEM}_{\pl}$.
\end{lemma}
\begin{proof}
We assume $F_{zz}x\land zz\approx yy$ and $\neg F_{yy}x$
We obtain (a): $\forall z (z\prec zz \to Pz x )$ (b): $\forall y(P y x
\to \exists z (z\prec zz \land O z y ))$, (c): $\neg \forall z (z\prec
yy \to Pz x )\lor \neg \forall y(P y x \to \exists z (z\prec yy \land  O
z y ))$. If $\neg \forall z (z\prec yy \to Pz x )$ then we take
$c/z\colon c\prec yy \land \neg Pcx$, so $c\prec zz$ by $zz\approx yy$
but then using (a) we have $Pcx$ which is false. Therefore, we have
$\forall z (z\prec yy \to Pz x )$ so using (c) we obtain $\neg \forall
y(P y x \to \exists z (z\prec yy \land  O z y ))$ we take $d/y\colon Pdx
\land \forall z (z\prec yy \to \neg O z d)$. From $Pdx$ and (b) we have
that $\exists z (z\prec zz \land  O z d )$ so we take $e/z\colon$
$e\prec zz \land Oed$. From  $e\prec zz$ and  $zz\approx yy$ we have
$e\prec yy$ thus using $\forall z (z\prec yy \to \neg O z d)$ we obtain $\neg Oed$
which yields a contradiction.
\end{proof}
As we see, all we needed was $\approx$ and logical axioms.
Now we show that in $\mathsf{AEM}_{\pl}$ with appropriate definitions everything is a unique
composition of unique atoms:
\begin{lemma}\label{6}
\eqref{AC1} is provable in $\mathsf{AEM}_{\pl}+\eqref{aaP}+\eqref{atP}$.
\end{lemma}
\begin{proof}
From Varzi’s analysis we have \eqref{TA1F} i.e. $F_{\at_{x}}x$ and
$\at_{x}\are aa$. The uniqueness of a mereological sum holds in extensional
mereology, as we note in the section \ref{1}, thus $\forall y
(F_{\at_{x}}y\to x=y)$. From $F_{\at_{x}}x$ and lemma \ref{lemextF} we obtain $\forall_{yy\are
aa}(yy\approx\at_{x} \to F_{yy}x)$, thus all we need to show is $\forall_{yy\are
aa}(F_{yy}x\to yy\approx\at_{x})$. We assume indirectly that $\neg \forall_{yy\are aa}(F_{yy}x\to yy\approx\at_{x})$
take $cc/yy$ and obtain $cc\are aa \land F_{cc}x\land cc\not = \at_{x}$.
 From $cc\not = \at_{x}$ and \eqref{dfapprox} we have $\exists z (z\prec
cc \land \neg z\prec \at_{x}\lor \neg z\prec cc \land z\prec \at_{x})$.
We take $c/z$ and we have two possibilities (a): $c\prec cc \land \neg
c\prec \at_{x}$ or (b): $\neg c\prec cc \land c\prec \at_{x}$. We start
with (a). From  $F_{cc}x$ and the definition of mereological sum
\eqref{defF} we obtain $\forall z (z\prec cc\to Pzx)$ thus using $c\prec
cc$ we have $Pcx$. From $F_{\at_{x}}x$ and \eqref{defF} we have $\forall
y(P y x \to \exists z (z\prec \at_{x} \land  O z y )$ so using $Pcx$ we
have $\exists z (z\prec \at_{x} \land  O z c )$ we take $d/z:$ $d\prec
\at_{x} \land Odc$. Both $c$ and $d$ are atoms, so from $Odc$ with the
use of \eqref{dfaa} we have $c=d$. Thus from $d\prec \at_{x}$ we obtain
$c\prec \at_{x}$ which yields a contradiction. In the case of (b) from
$c\prec \at_{x}$ and \eqref{dfatP} we obtain $Pcx$. Next from $F_{cc}x$
and definition of sum \eqref{defF} we obtain $\forall y (Pyx\to \exists
z (z\prec cc\land Ozy))$. We take $c/z$ and using $Pcx$ we obtain
$\exists z (z\prec cc\land Ozc)$ We take $d/z$ and so $d\prec cc\land
Odc$. We know that $c$ is an atom because of $c\prec \at_{x}$, and $d$
is an atom because of $d\prec cc$ and $cc\are aa$, so $Ocd$ yields $c=d$
and this with $d\prec cc$ yields $c\prec cc$ which yields a contradiction
with assumption.
\end{proof}
To prove \eqref{AC2} in $\mathsf{AEM}_{\pl}$ we need to introduce the
notion of atoms of $\pmb{xx}$'s and we take the following instantiation
of comprehension schema
\begin{equation}
\tag{$\mathtt{Df}.\at_{P}^{\pl}$} \label{dfatP} y\prec \at_{xx}\leftrightarrow
\exists_{z\are xx}(P y z \land y\prec aa).
\end{equation}
Now we prove that, in atomistic extensional mereology with plural
quantification, being the mereological sum of some objects is equivalent
with mereological sum of atoms of these objects. Implication
`$\leftarrow$' requires the use of \eqref{ssp}.
\begin{lemma}
$F_{\at_{zz}}x\to F_{zz}x$ is provable in
$\mathsf{AEM}_{\pl}+\eqref{aaP}+\eqref{atP}+\eqref{dfatP}$.
\end{lemma}
\begin{proof}
We assume $F_{\at_{zz}}x$ and indirectly we assume $\neg F_{zz}x$. We
obtain the following
(a): $\forall z (z\prec \at_{zz} \to Pzx)$, (b): $\forall y (Pyx\to
\exists z (z\prec \at_{zz} \land  Ozy))$
and (c): $\neg \forall z (z\prec zz\to Pzx ) \lor \neg \forall y (Pyx\to
\exists z (z\prec zz \land Ozy))$. If it were the case that $\neg \forall z
(z\prec zz\to Pzx )$ then we take $c/z$ and we have $c\prec zz\land \neg
Pcx$. From $\neg Pcx$ and \eqref{ssp} we obtain that there is $d$ such
that $Pd c\land \neg Odx$. From \eqref{TA2P} there is an atom of $d$:
$Ae\land Ped$. From the latter and $Pdc$ using \eqref{transing} we have
$Pec$. So, we have $c\prec zz \land Pec\land Ae$ and thus using
\eqref{dfatP} we obtain $e\prec\at_{zz}$. The latter with (a) yields
$Pex$, which is false because we have $Ped \land \neg Odx $. So $\forall
z (z\prec zz\to Pzx )$ and next by (c): we obtain $\neg \forall y
(Pyx\to \exists z (z\prec zz \land Ozy))$. We take $f/y$ and so
$Pfx\land \forall z (z\prec zz \to \neg Ozf)$. From $Pfx$ and (b) we
have $\exists z(z\prec \at_{zz}\land Ozf)$ and we obtain $g/z: g\prec
\at_{zz}\land Ofg$. From $g\prec \at_{zz}$ and \eqref{atP} we obtain that
there is $h\prec zz$ such that $Pgh\land Ag$. We have $Ofg$ so with $Ag$
using definitions \eqref{dfAP} and \eqref{defO} we obtain $Pgf$. The
latter, combined with $Pgh$ which was obtained earlier yields $Ohf$. We have $h\prec zz$ so
using $\forall z (z\prec zz \to \neg Ozf)$ we obtain $\neg Ohf$ which
yields a contradiction.
\end{proof}
\begin{lemma}\label{fuzzatzz}
$F_{zz}x\to F_{\at_{zz}}x$ is provable in
$\mathsf{AEM}_{\pl}+\eqref{aaP}+\eqref{atP}+\eqref{dfatP}$.
\end{lemma}
\begin{proof}
We proceed indirectly. We assume $F_{zz}x$ and $\neg F_{\at_{zz}}x$ and
we obtain  (a): $\forall z (z\prec zz \to Pzx)$, (b): $\forall y (Pyx
\to \exists z (z\prec zz \land O zy))$ and also we have (c): $\neg
\forall z (z\prec \at_{zz}\to Pzx ) \lor \neg \forall y (Pyx\to \exists
z (z\prec \at_{zz} Ozy))$. If it were the case that $\neg \forall z (z\prec
\at_{zz}\to Pzx )$ then we take $c/x$ and so $c\prec \at_{zz}\land \neg 
Pcx$. From $c\prec \at_{zz}$ and definition of \eqref{dfatP} we have
that there is $d\prec zz$ with  $Pcd \land Ac$. From (a) and $d\prec zz$
we obtain $Pdx$, so using $Pcd$ and \eqref{transing} we obtain $Pcx$,
which is false, so $\forall z (z\prec \at_{zz}\to Pzx)$ and next by (c)
$\neg \forall y (Pyx\to \exists z (z\prec \at_{zz} \land Ozx))$. We take
$e/y\colon Pex\land \forall z (z\prec \at_{zz} \to \neg Oze)$. From
$Pex$ and (b) we obtain $\exists z (z\prec zz \land O ze)$ We take $f/z$
and so $f\prec zz \land O fe$. From $O fe$ we obtain $Pge\land Pgf$.
 From \eqref{TA2P} we obtain that
there is an atom of $g$: $Ah \land Phg$ from \eqref{transing} we obtain
$Phe\land Phf$. We have $Phf\land f\prec zz \land Ah$ and by using
\eqref{dfatP} we get $h\prec \at_{zz}$. The latter with $\forall z
(z\prec \at_{zz} \to \neg Oze)$ yields $\neg Ohe$, which yields a
contradiction with $Phe$.
\end{proof}
To end the proof we need to note a few more things. The notions of being
a part \eqref{dfP}, atoms \eqref{dfaa}, and atoms of: \eqref{defatx} and \eqref{dfpl} in
$\ATC$ are theses of $\mathsf{AEM}_{pl}+\eqref{dfatP}$. \eqref{dfP} in
one direction follows just from \eqref{transing} and conversely we
have to use \eqref{ssp} and \eqref{TA2P}. \eqref{dfaa} follows from
the defined mereological sum \eqref{defF}, being an atom \eqref{aaP} and the fact $\forall
x (F_{\at_{x}}x)$ proved by Varzi. We have proved \eqref{defatx} in
lemma \ref{6}, and \eqref{dfpl} can be easily proved with the use of
the definition of atom \eqref{dfAP}, atoms of \eqref{dfatP}, mereological sum \eqref{defF} and the fact that in
$\mathsf{AEM_{\pl}}+\eqref{aaP}+\eqref{atP}$ we have $\forall x (F_{\at_{x}}x)$. Lastly, we note
\eqref{dfatP} is a thesis of $\ATC$, and it can be proved with
the use of \eqref{defatx}, \eqref{defF}, \eqref{dfpl}, and \eqref{dfP}.
Thus, using lemmas \ref{lemextF}-\ref{fuzzatzz} we obtain that
axiomatic theory of atomism $\ATC$ presented herein, when extended by appropriate
definitions is a subtheory of atomistic extensional mereology:
\begin{theorem}\label{ACAEM}
$\ATC+\eqref{dfP}+\eqref{AF}\subseteq \mathsf{AEM}_{\pl}+\eqref{aaP}+\eqref{atP}+\eqref{dfatP}$.
\end{theorem}
Using theorems \ref{AEMAC} and \ref{ACAEM} we obtain the final result of the
work that out theory and atomistic extensional mereology extended with appropriate instantiations of comprehension schema are definitionally equivalent
$$\ATC+\eqref{dfP}+\eqref{AF}=\mathsf{AEM}_{\pl}+\eqref{aaP}+\eqref{atP}+\eqref{dfatP}.$$
As we mentioned, in extensional mereology fusion \eqref{deffu} and sum \eqref{defF} are equivalent. So $\ATC+\eqref{dfP}+\eqref{AF}$ is also equivalent to $\mathsf{AEM}_{\pl}+\eqref{aaP}+\eqref{atP}+\eqref{dfatP}$ with fusion instead of sum. 

Perhaps some atomists would prefer to stay in first-order logic without plural quantification.
However, plural quantification in the case of atomism has its
advantages. First, it is open to the axiomatic characterization of superatomism
\cite{Di20}. Second, as we have shown, atomism does not need to be based
on the primitive notion of being a part. Finally, we believe that it captures both
formally and intuitively what \eqref{TA1} is meant to express.

In conclusion, we also wish to emphasize that though our axiomatic system $\ATC$ for atomism is equivalent to atomistic extensional mereology with plural quantification, $\ATC$ it is based on two axioms that cannot be proven in any formulation of \textit{pure} extensional mereology. However, general exstensonal mereology can also be axiomatized, in plural logic, using fusion as the primitive concept \cite{VaLy22}.
%
\begin{small}

\end{small}
\end{document}